\documentclass[12pt,reqno]{amsart}

\usepackage{fancyhdr}
\newtheorem{theorem}{Theorem}[section]

\usepackage{graphics}
\usepackage{amssymb}
\usepackage{cite}
\usepackage{mathrsfs}
\usepackage{amsmath}

\newtheorem{remark}[theorem]{Remark}

\numberwithin{equation}{section}

\large\normalsize
\title{}

\begin{document}

\title{On the edge reconstruction of six digraph polynomials}

\author{Jingyuan Zhang, Xian'an Jin}
\address{ School of Mathematical Science, Xiamen University, Xiamen 361001, China.
}

\author{Weigen Yan$^{*}$}

\address{$*$Corresponding author,
School of Science, Jimei University, Xiamen 361021, China
}

\email{doriazhang@outlook.com, xajin@xmu.edu.cn, weigenyan@263.net
}
\thanks{The work was supported in part by NSFC
Grants (12171402, 12071180).
}
\thanks{ 
}


\begin{abstract}
Let $G=(V,E)$ be a digraph having no loops and no multiple arcs, with vertex set $V=\{v_1,v_2,\ldots,v_n\}$ and arc set $E=\{e_1,e_2,\ldots,e_m\}$. Denote the adjacency matrix and the vertex in-degree diagonal matrix of $G$ by $A=(a_{ij})_{n\times n}$ and $D=diag(d^+(v_1),d^+(v_2),\cdots,d^+(v_n))$, where $a_{ij}=1$ if $(v_i,v_j)\in E(G)$ and $a_{ij}=0$ otherwise, and $d^+(v_i)$ is the number of arcs with head $v_i$. Set $f_1(G;x)=\det(xI-A), f_2(G;x)=\det(xI-D+A),f_3(G;x)=\det(xI-D-A),f_4(G;x)={\rm per}(xI-A), f_5(G;x)={\rm per}(xI-D+A),f_6(G;x)={\rm per}(xI-D-A)$, where $\det(X)$ and ${\rm per}(X)$ denote the determinant and the permanent of a square matrix $X$, respectively. In this paper, we consider a variant of the Ulam's vertex reconstruction conjecture and the Harary's edge reconstruction conjecture, and prove that, for any $1\leq i\leq 6$,
\begin{equation*}
(m-n)f_i(G;x)+xf_i'(G;x)=\sum\limits_{e\in E}f_i(G-e;x),
\end{equation*}
 which implies that if $m\neq n$, then $f_i(G;x)$ can be reconstructed from $\{f_i(G-e;x)|e\in E\}$.
\end{abstract}

\maketitle

\section{Introduction}
\hspace*{\parindent}
The famous Ulam's conjecture \cite{Ulam60} states that if two simple graphs $G_1$ and $G_2$ with vertex sets $V(G_1)=\{u_1,u_2,\ldots,u_n\}$ and $V(G_2)=\{v_1,v_2,\ldots,v_n\}$ and $n\geq 3$ satisfy that for each $1\leq i\leq n$, $G_1-u_i$ and $G_2-v_i$ are isomorphic, then $G_1$ and $G_2$ are isomorphic. That is, each simple graph $G$ with at least three vertices can be uniquely reconstructed from its deck $\{G-v|v\in V(G)\}$. Although some reconstructible graphs \cite{HM69} are obtained, the Ulam's conjecture is still open.

Harary \cite{Har65} posed a similar conjecture (i.e., the edge reconstruction conjecture), which states that every simple graph $G$ with edge set $E(G)$ can be reconstructed from its edge deck $\{G-e|e\in E(G)\}$ if $|E(G)|\geq 4$. Lov\'asz \cite{Lov72} proved that the edge reconstruction conjecture holds for simple graphs with $n$ vertices and at least $\frac{n(n-1)}{4}$ edges. M\"{u}ller \cite{Mul77} improved the Lov\'asz's result and proved that the edge reconstruction conjecture holds for simple graphs with $n$ vertices and more than $n\cdot\log_2n$ edges. Although M\"uller's result implies that the edge reconstruction conjecture holds for almost all simple graphs, the Harary's conjecture is still open.

Cvetkovi\'c, at the XVIII International Scientific Colloquium in Ilmenau in 1973, posed a related problem as follows: Can the characteristic polynomial $f(G;x)$ of a simple graph $G$ with vertex set $V(G)$ be reconstructed from $\{f(G-v;x)|v\in V(G)\}$ for $|V(G)|\geq 3$? The same problem was independently posed by Schwenk \cite{Sch79}. Gutman and Cvetkovi\'c \cite{GC75} obtained some results related to this problem. Note that $f'(G;x)=\sum\limits_{v\in V(G)}f(G-v;x)$. Hence the coefficients of $f(G;x)$ except for the constant term can be reconstructed from $\{f(G-v;x)|v\in V(G)\}$. No examples of non-unique reconstruction of the characteristic polynomial of graphs are known.  Under the assumption that the reconstruction of the characteristic polynomial is not unique, Cvetkovi\'c \cite{Cvet00} described some properties of graphs $G$ such that the constant term of $f(G;x)$ can not be reconstructed from $\{f(G-v;x)|v\in V(G)\}$. For the reconstruction problem of the characteristic polynomial of graphs, see for example the nice survey \cite{SS23}.

For the reconstructing problems of graphs, see for example some references \cite{And82,BIK05,For04,Godsil87,GC75,GM81,Hag00,Hos22,Man76,Thom77,Yuan82}.

In this paper, we consider a variant of the three problems above---the edge reconstruction problems on six digraph polynomials (see Egs. (1.1)-(1.6)). In the following, we assume that $G=(V(G),E(G))$ is a digraph having no loops and no multiple arcs, with vertex set $V(G)=\{v_1,v_2,\ldots,v_n\}$ and arc set $E(G)=\{e_1,e_2,\ldots,e_m\}$. Denote the adjacency matrix and the vertex in-degree diagonal matrix of $G$ by $A=(a_{ij})_{n\times n}$ and $D=diag(d^+(v_1),d^+(v_2),\cdots,d^+(v_n))$, where $a_{ij}=1$ if $(v_i,v_j)\in E(G)$ and $a_{ij}=0$ otherwise, and $d^+(v_i)$ is the number of arcs with head $v_i$. Then $D-A$ and $D+A$ are the Laplacian and the signless Laplacian matrices of $G$. Set
\begin{align}
f_1(G;x)&=\det(xI-A),\\
f_2(G;x)&=\det(xI-D+A),\\
f_3(G;x)&=\det(xI-D-A),\\
f_4(G;x)&={\rm per}(xI-A),\\
f_5(G;x)&={\rm per}(xI-D+A),\\
f_6(G;x)&={\rm per}(xI-D-A),
\end{align}
where $\det(X)$ and ${\rm per}(X)$ denote the determinant and permanent of a square matrix $X$. Then $f_1(G;x), f_2(G;x),f_3(G;x),f_4(G;x), f_5(G;x)$, and $f_6(G;x)$ are called the characteristic polynomial, Laplacian characteristic polynomial, signless Laplacian characteristic polynomial, permanental polynomial, Laplacian permanental polynomial, signless Laplacian permanental polynomial of the digraph $G$, respectively.

In the next section, we obtain two identities, one is is related to the determinants, and the other is related to the permanents. Using these two identities, in Section 3,  we prove that, for any $1\leq i\leq 6$,
\begin{equation}
(m-n)f_i(G;x)+xf_i'(G;x)=\sum\limits_{e\in E(G)}f_i(G-e;x).
\end{equation}
In Section 4, we prove that, $f_2(G;x)$ can be reconstructed from $\{f_2(G-e;x)|e\in E(G)\}$, and for any $i=1,3,4,5,6$, if $m\neq n$, then $f_i(G;x)$ can be reconstructed from $\{f_i(G-e;x)|e\in E(G)\}$, and for $i=1,4$, if $m=n$, some counterexamples are given to show that $f_i(G;x)$ can not be determined by $\{f_i(G-e;x)|e\in E(G)\}$.

\section{Two identities}
Let $X=(x_{st})_{n\times n}$ be a matrix of order $n$ over the complex field. For any $1\leq i,j\leq n$, define a matrix $X_{ij}=(x_{st}^{ij})_{n\times n}$, where
\begin{displaymath}
x_{st}^{ij} = \left\{ \begin{array}{ll}
x_{st}, & \textrm{if}\ (s,t)\neq (i,j); \\
0, & \textrm{if }\ (s,t)=(i,j).
\end{array} \right.
\end{displaymath}
That is, $X_{ij}$ is the matrix obtained from the matrix $X$ by replacing the $(i,j)$-entry $x_{ij}$ with $0$. For example, if
$$X=\left(\begin{array}{ccc}
a & b & c\\
e & f & g\\
r & s & t
\end{array}
\right),
$$
then
$$X_{12}=\left(\begin{array}{ccc}
a & 0 & c\\
e & f & g\\
r & s & t
\end{array}
\right),
X_{33}=\left(\begin{array}{ccc}
a & b & c\\
e & f & g\\
r & s & 0
\end{array}
\right).
$$
Obviously, if $x_{ij}=0$, then $X=X_{ij}$. Now we can prove the following result which will play an important role in the proof of the main results in this paper.

\begin{theorem}
Let $X=(x_{st})_{n\times n}$ be a matrix of order $n$ over the complex field and let $X_{ij}$ be defined as above.  Then the determinant of $X$ satisfies:
\begin{equation}
(n^2-n)\det(X)=\sum_{1\leq i,j\leq n}\det(X_{ij}).
\end{equation}
\end{theorem}
\begin{proof}
For convenience, we assume that $\{x_{11},x_{12},\ldots,x_{nn}\}$ is a set of $n^2$ independent commuting variables over the complex field. Hence there exist $n!$ terms in the expansion of the determinant $\det(X)$, denoted by $\mathcal T_{\alpha_1}, \mathcal T_{\alpha_2}, \ldots, \mathcal T_{\alpha_{n!}}$, where $S_n=\{\alpha_i|1\leq i\leq n\}$ is the symmetric group of order $n$ and $\mathcal T_{\alpha_i}=sgn(\alpha_i)x_{1\alpha_i(1)}x_{2\alpha_i(2)}\ldots x_{n\alpha_i(n)}$. Hence
\begin{equation}
\det(X)=\sum_{i=1}^{n!}\mathcal T_{\alpha_i}.
\end{equation}

We define an edge-weighted bipartite graph $\mathcal G(X)$ with bipartition $(U,V)$, where $U=S_n=\{\alpha_1,\alpha_2,\ldots,\alpha_{n!}\}$ and $V=\{x_{ij}|1\leq i,j\leq n\}$, and $(\alpha_i,x_{st})$ is an edge of $\mathcal G(X)$ if and only if $(s,t)\notin \{(k,\alpha_i(k))|1\leq k\leq n\}$, i.e., $(\alpha_i,x_{st})$ is an edge of $\mathcal G(X)$ if and only if $x_{st}$ is not a factor of the term $\mathcal T_{\alpha_i}=sgn(\alpha_i)x_{1\alpha_i(1)}x_{2\alpha_i(2)}\ldots x_{n\alpha_i(n)}$ in the expansion of $\det(X)$. The edge weight function $\omega$ of $\mathcal G(X)$ satisfies: for any edge $(\alpha_i,x_{st})$ of $\mathcal G(X)$, $\omega((\alpha_i,x_{st}))=\mathcal T_{\alpha_i}$. Denote by $T=(t_{\alpha_ix_{st}})_{n!\times n^2}$ the bipartite adjacency matrix of $\mathcal G(X)$. That is,
for any $1\leq i\leq n!, 1\leq s,t\leq n$,
\begin{displaymath}
t_{\alpha_ix_{st}} = \left\{ \begin{array}{ll}
\mathcal T_{\alpha_i}, & \textrm{if}\ (s,t)\notin \{(k,\alpha_i(k))|1\leq k\leq n\}; \\
0, &  \textrm{otherwise}.
\end{array} \right.
\end{displaymath}
In other words, each entry in the row indexed by $\alpha_i$ in matrix $T$ equals $\mathcal T_{\alpha_i}$ if $(s,t)\notin \{(k,\alpha_i(k))|1\leq k\leq n\}$, and equals zero otherwise. Hence each row of $T$ has $n^2-n$ entries each of which equals $\mathcal T_{\alpha_i}$ and $n$ entries each of which equals zero.

For example, if $n=3$, then $U=S_3=\{\alpha_1,\alpha_2,\alpha_3,\alpha_4,\alpha_5,\alpha_6\}$ and $V=\{x_{11},x_{12},x_{13},x_{21},\\x_{22},x_{23},x_{31},x_{32},x_{33}\}$, where
\begin{equation*}
\alpha_1=\left(
\begin{array}{ccc}
1&2&3\\
1&2&3
\end{array}
\right),
\alpha_2=\left(
\begin{array}{ccc}
1&2&3\\
1&3&2
\end{array}
\right),
\alpha_3=\left(
\begin{array}{ccc}
1&2&3\\
3&2&1
\end{array}
\right),
\end{equation*}
\begin{equation*}
\alpha_4=\left(
\begin{array}{ccc}
1&2&3\\
2&1&3
\end{array}
\right),
\alpha_5=\left(
\begin{array}{ccc}
1&2&3\\
2&3&1
\end{array}
\right),
\alpha_6=\left(
\begin{array}{ccc}
1&2&3\\
3&1&2
\end{array}
\right).
\end{equation*}
Then $\mathcal T_{\alpha_1}=x_{11}x_{22}x_{33}, \mathcal T_{\alpha_2}=-x_{11}x_{23}x_{32},\mathcal T_{\alpha_3}=-x_{13}x_{22}x_{31},\mathcal T_{\alpha_4}=-x_{12}x_{21}x_{33},\mathcal T_{\alpha_5}=x_{12}x_{23}x_{31}, \mathcal T_{\alpha_6}=x_{13}x_{21}x_{32}$,
and
\begin{equation*}
T=\left(
\begin{array}{cccccccccc}
0 & \mathcal T_{\alpha_1}& \mathcal T_{\alpha_1}& \mathcal T_{\alpha_1}& 0& \mathcal T_{\alpha_1} & \mathcal T_{\alpha_1}  & \mathcal T_{\alpha_1} & 0\\
0 & \mathcal T_{\alpha_2} &  \mathcal T_{\alpha_2} & \mathcal T_{\alpha_2} & \mathcal T_{\alpha_2} & 0 & \mathcal T_{\alpha_2} & 0 & \mathcal T_{\alpha_2}\\
\mathcal T_{\alpha_3} & \mathcal T_{\alpha_3} & 0 &  \mathcal T_{\alpha_3} & 0 &  \mathcal T_{\alpha_3} & 0 & \mathcal T_{\alpha_3}  & \mathcal T_{\alpha_3}\\
\mathcal T_{\alpha_4} & 0 &  \mathcal T_{\alpha_4} & 0 &  \mathcal T_{\alpha_4} & \mathcal T_{\alpha_4} &  \mathcal T_{\alpha_4} & \mathcal T_{\alpha_4} & 0\\
\mathcal T_{\alpha_5} & 0 & \mathcal T_{\alpha_5} & \mathcal T_{\alpha_5} & \mathcal T_{\alpha_5} & 0 & 0 & \mathcal T_{\alpha_5} & \mathcal T_{\alpha_5}\\
\mathcal T_{\alpha_6} &  \mathcal T_{\alpha_6} & 0 & 0 & \mathcal T_{\alpha_6} & \mathcal T_{\alpha_6} & \mathcal T_{\alpha_6} & 0 & \mathcal T_{\alpha_6}
\end{array}
\right).
\end{equation*}

Note that the $i$-th row in matrix $T$ has exactly $n^2-n$ non-zero entries each of which equals $\mathcal T_{\alpha_i}$. Hence the sum of entries of $i$-th row in $T$ equals $(n^2-n)\mathcal T_{\alpha_i}$. So the sum of all entries in $T$ equals $(n^2-n)\det(X)$. That is,
\begin{equation}
\sum_{i=1}^{n!}\sum_{1\leq s,t\leq n}t_{\alpha_ix_{st}}=(n^2-n)\det(X).
\end{equation}

On the other hand, by the definition of $X_{st}$, the sum of entries of $x_{st}$-th column of $T$ equals exactly the determinant of $X_{st}$. Hence
\begin{equation}
\sum_{1\leq s,t\leq n}\sum_{i=1}^{n!}t_{\alpha_ix_{st}}=\sum_{1\leq s,t\leq n}\det(X_{st}).
\end{equation}

Eq. (2.1) is immediate from Eqs. (2.3) and (2.4) and hence the theorem holds.
\end{proof}

In the theorem above, let $m$ be the number of non-zero entries of $X$. If $x_{ij}=0$, then $X_{ij}=X$. The following result is equivalent to the theorem above.

\begin{theorem}
Let $X=(x_{st})_{n\times n}$ be a matrix of order $n$ over the complex field and let $m$ be the number of non-zero entries of $X$. Then
\begin{equation}
(m-n)\det(X)=\sum_{(i,j)\in I}\det(X_{ij}),
\end{equation}
where $I=\{(i,j)|x_{ij}\neq 0, 1\leq i,j\leq n\}$.
\end{theorem}

Note that the permanent of a matrix $X=(x_{ij})_{n\times n}$ is defined as
\begin{equation}
{\rm per}(X)=\sum_{\alpha\in S_n}x_{1\alpha(1)}x_{2\alpha(2)}\ldots x_{n\alpha(n)},
\end{equation}
where $\alpha$ ranges over the set of the symmetric group of order $n$. Similarly, we can prove the following result.

\begin{theorem}
Let $X=(x_{st})_{n\times n}$ be a matrix of order $n$ over the complex field and let $m$ be the number of non-zero entries of $X$. Then the permanent ${\rm per}(X)$ of $X$ satisfies:
\begin{equation}
(m-n){\rm per}(X)=\sum_{(i,j)\in I}{\rm per}(X_{ij}),
\end{equation}
where $I=\{(i,j)|x_{ij}\neq 0, 1\leq i,j\leq n\}$.
\end{theorem}

\section{Six digraph polynomials }
In this section, we prove firstly a result which implies each of the six digraph polynomials $f_1(G;x),f_2(G;x),\ldots,f_6(G;x)$ defined in Section 1 satisfies Eq. (1.7).

Suppose that both $\beta$ and $\gamma$ are real numbers satisfying $\gamma\neq0$. Let $G$ be a digraph with vertex set $V(G)=\{v_1,v_2,\ldots,v_n\}$ and arc set $E(G)=\{e_1,e_2,\ldots,e_m\}$, which has no loops or multiple arcs. Set
$$g_1(G;x)=\det(xI_n-\beta D-\gamma A)$$
and
$$g_2(G;x)={\rm per}(xI_n-\beta D-\gamma A),$$
where $D$ and $A$ are the vertex in-degree diagonal matrix and the adjacency matrix of $G$, respectively.

\begin{theorem}
Both $g_1(G;x)$ and $g_2(G;x)$ defined as above satisfy:
\begin{equation}
(m-n)g_1(G;x)+xg_1'(G;x)=\sum_{e\in E(G)}g_1(G-e;x),
\end{equation}
\begin{equation}
(m-n)g_2(G;x)+xg_2'(G;x)=\sum_{e\in E(G)}g_2(G-e;x).
\end{equation}
\end{theorem}
\begin{proof}
Note that $G$ has $m$ arcs. Hence $xI_n-\beta D-\gamma A$ has $m+n$ non-zero entries. By Theorems 2.2 and $2.3$,
\begin{equation}
(m+n-n)\det(xI_n-\beta D-\gamma A)=\sum_{e\in E(G)}\det(xI_n-\beta D-\gamma A_e)+\sum_{i=1}^n \det(xI_n^{(i)}-\beta D^{(i)}-\gamma A),
\end{equation}
\begin{equation}
(m+n-n){\rm per}(xI_n-\beta D-\gamma A)=\sum_{e\in E(G)}{\rm per}(xI_n-\beta D-\gamma A_e)+\sum_{i=1}^n {\rm per}(xI_n^{(i)}-\beta D^{(i)}-\gamma A),
\end{equation}
where $A_e$ is the adjacency matrix of digraph $G-e$, and $I_n^{(i)}$ is the diagonal matrix of order $n$ with diagonal entries equal to one except for the $i$-th entry equal to zero, and $D^{(i)}=diag(d_G^+(v_1),\ldots,d_G^+(v_{i-1}),0,d_G^+(v_{i+1}),\ldots,d_G^+(v_n))$, and $d_G^+(v)$ denotes the number of arcs with head $v$ in G.

Denote the vertex in-degree diagonal matrix of $G-e$ by $D_e$. Without loss of generality, let $e=(v_s,v_t)$. Then
\begin{equation}
\det(xI_n-\beta D-\gamma A_e)=\det(xI_n-\beta D_e-\gamma A_e)-\beta\det[(xI_n-\beta D-\gamma A)_t],
\end{equation}
\begin{equation}
{\rm per}(xI_n-\beta D-\gamma A_e)={\rm per}(xI_n-\beta D_e-\gamma A_e)-\beta {\rm per}[(xI_n-\beta D-\gamma A)_t],
\end{equation}
where $(xI_n-\beta D-\gamma A)_t$ is the matrix obtained from $xI_n-\beta D-\gamma A$ by deleting the $t$-th row and the $t$-th column.

On the other hand,
\begin{equation}
\det(xI_n^{(i)}-\beta D^{(i)}-\gamma A)=\det(xI_n-\beta D-\gamma A)-(x-\beta d_G^+(v_i))\det[(xI_n-\beta D-\gamma A)_i],
\end{equation}
\begin{equation}
{\rm per}(xI_n^{(i)}-\beta D^{(i)}-\gamma A)={\rm per}(xI_n-\beta D-\gamma A)-(x-\beta d_G^+(v_i)){\rm per}[(xI_n-\beta D-\gamma A)_i].
\end{equation}

Hence
\begin{align}
&\sum_{e\in E(G)}\det(xI_n-\beta D-\gamma A_e)\nonumber\\
&=\sum_{e\in E(G)}\det(xI_n-\beta D_e-\gamma A_e)-\beta\sum_{(v_s,v_t)\in E(G)}\det[(xI_n-\beta D-\gamma A)_t]\nonumber\\
&=\sum_{e\in E(G)}g_1(G-e;x)-\beta\sum_{t=1}^nd_G^+(v_t)\det[(xI_n-\beta D-\gamma A)_t],
&\\
\nonumber\\
&\sum_{e\in E(G)}{\rm per}(xI_n-\beta D-\gamma A_e)\nonumber\\
&=\sum_{e\in E(G)}{\rm per}(xI_n-\beta D_e-\gamma A_e)-\beta\sum_{(v_s,v_t)\in E(G)}{\rm per}[(xI_n-\beta D-\gamma A)_t]\nonumber\\
&=\sum_{e\in E(G)}g_2(G-e;x)-\beta\sum_{t=1}^nd_G^+(v_t){\rm per}[(xI_n-\beta D-\gamma A)_t],
&\\
\nonumber\\
&\sum_{i=1}^n\det(xI_n^{(i)}-\beta D^{(i)}-\gamma A)\nonumber\\
&=\sum_{i=1}^n\left[\det(xI_n-\beta D-\gamma A)-(x-\beta d_G^+(v_i))\det[(xI_n-\beta D-\gamma A)_i]\right]\nonumber\\
&=ng_1(G;x)-x\sum_{i=1}^n\det[(xI_n-\beta D-\gamma A)_i]+\beta\sum_{i=1}^nd_G^+(v_i)\det[(xI_n-\beta D-\gamma A)_i],
\end{align}
\begin{align}
&\sum_{i=1}^n{\rm per}(xI_n^{(i)}-\beta D^{(i)}-\gamma A)\nonumber\\
&=\sum_{i=1}^n\left[{\rm per}(xI_n-\beta D-\gamma A)-(x-\beta d_G^+(v_i)){\rm per}[(xI_n-\beta D-\gamma A)_i]\right]\nonumber\\
&=ng_2(G;x)-x\sum_{i=1}^n{\rm per}[(xI_n-\beta D-\gamma A)_i]+\beta\sum_{i=1}^nd_G^+(v_i){\rm per}[(xI_n-\beta D-\gamma A)_i].
\end{align}

It is not difficult to show that
\begin{equation}
\sum_{i=1}^n\det[(xI_n-\beta D-\gamma A)_i]=g_1'(G;x),
\end{equation}
\begin{equation}
\sum_{i=1}^n{\rm per}[(xI_n-\beta D-\gamma A)_i]=g_2'(G;x).
\end{equation}
By Eqs. (3.9)-(3.14),
\begin{align}
&\sum_{e\in E(G)}\det(xI_n-\beta D-\gamma A_e)+\sum_{i=1}^n\det(xI_n^{(i)}-\beta D^{(i)}-\gamma A)\nonumber\\
&=\sum_{e\in E(G)}g_1(G-e;x)+ng_1(G;x)-xg_1'(G;x),
\end{align}
\begin{align}
&\sum_{e\in E(G)}{\rm per}(xI_n-\beta D-\gamma A_e)+\sum_{i=1}^n{\rm per}(xI_n^{(i)}-\beta D^{(i)}-\gamma A)\nonumber\\
&=\sum_{e\in E(G)}g_2(G-e;x)+ng_2(G;x)-xg_2'(G;x).
\end{align}
So the theorem is immediate from Eqs. (3.3), (3.4), (3.15) and (3.16).
\end{proof}

\begin{remark}
Theorem 3.1 still holds for any arc-weighted digraph $G$ with arc-weight function $\omega:E(G)\rightarrow\mathbb{R}\setminus\{0\}$, where $E(G)$ is the arc set of $G$.
\end{remark}

\begin{theorem}
Let $G$ be a digraph with vertex set $V(G)=\{v_1,v_2,\ldots,v_n\}$ and arc set $E(G)=\{e_1,e_2,\ldots,e_m\}$, which has no loops or multiple arcs. Then each of the six digraph polynomials $f_1(G;x),f_2(G;x),\ldots,f_6(G;x)$ defined in Section 1 satisfies Eq. {\rm (1.7)}, i.e.,
\begin{equation*}
(m-n)f_i(G;x)+xf_i'(G;x)=\sum_{e\in E(G)}f_i(G-e;x)
\end{equation*}
for any $1\leq i\leq6$.
\end{theorem}
\begin{proof}
In Theorem 3.1, if we set $\beta=0$ and $\gamma=1$, then both $f_1(G;x)$ and $f_4(G;x)$ satisfy Eq. (1.7), if we set $\beta=1$ and $\gamma=-1$, then both $f_2(G;x)$ and $f_5(G;x)$ satisfy Eq. (1.7), and if we set $\beta=1$ and $\gamma=1$, then both $f_3(G;x)$ and $f_6(G;x)$ satisfy Eq. (1.7).
\end{proof}

\section{The edge reconstruction of $f_1(G;.x),f_2(G;x),\ldots,f_6(G;x)$}

Using the results in the section above, in this section, we discuss the edge reconstruction problem: Can the six digraph polynomials $f_1(G;x),f_2(G;x),\ldots,f_6(G;x)$ defined in Section 1 be determined by $\{f_1(G-e;x)|e\in E(G)\}, \{f_2(G-e;x)|e\in E(G)\}, \ldots, \{f_6(G-e;x)|e\in E(G)\}$, respectively?

\begin{theorem}
Let $G$ be a digraph with vertex set $V(G)=\{v_1,v_2,\ldots,v_n\}$ and arc set $E(G)=\{e_1,e_2,\ldots,e_m\}$, which has no loops or multiple arcs.
If $m\neq n$, then $f_i(G;x)$ can be reconstructed from $\{f_i(G-e;x)|e\in E(G)\}$ for $1\leq i\leq 6$.
\end{theorem}
\begin{proof}
Note that, for any $i=1,2,\ldots,6$, by Theorems 3.3, $f_i(G;x)$ satisfies the differential equation Eq. (1.7). If $m\neq n$, then $f_i(G;0)=\frac{1}{m-n}\sum\limits_{e\in E(G)}f_i(G-e;0)$. Hence this differential equation has a unique solution, and the theorem holds.
\end{proof}

If $m=n$, we consider two digraphs $G_1$ and $G_2$ with vertex sets $V(G_1)=\{u_1,u_2,\ldots,u_n\},\\ V(G_2)=\{v_1,v_2,\ldots,v_n\}$ and arc sets $E(G_1)=\{(u_1,u_2), (u_2,u_3),\ldots,(u_{n-1},u_n),(u_n,u_1)\},\\ E(G_2)=\{(v_1,v_2), (v_2,v_3),\ldots,(v_{n-1},v_n),(v_1,v_n)\}$. Obviously,
$$f_1(G_1;x)=x^n-1,\ f_1(G_2;x)=x^{n},$$
$$f_4(G_1;x)=x^n+(-1)^n,\ f_4(G_2;x)=x^n,$$
and
$$f_1(G_1-a;x)=x^n,\ f_1(G_2-b;x)=x^n,$$
$$f_4(G_1-a;x)=x^n,\ f_4(G_2-b;x)=x^n$$
for any $a\in E(G_1), b\in E(G_2)$, and hence $$f_1(G_1;x)\neq f_1(G_2;x),\ f_4(G_1;x)\neq f_4(G_2;x).$$
But $$\{f_i(G_1-e;x)|e\in E(G_1)\}=\{f_i(G_2-e;x)|e\in E(G_2)\}=\left\{\overbrace{x^n,x^n,\ldots,x^n}^n\right\}$$ for $i=1,4$.

Hence, we have the following result.
\begin{theorem}
If $m=n$, the characteristic polynomial $f_1(G;x)$ and the permanental polynomial $f_4(G;x)$ of a digraph $G$ with $n$ vertices and $n$ arcs can not be determined uniquely by $\{f_1(G-e;x)|e\in E(G)\}$ and $\{f_4(G-e;x)|e\in E(G)\}$, respectively.
\end{theorem}

Note that, $f_2(G;x)=\det(xI-D+A)$. Hence $f_2(G;0)=\det(A-D)=0$. If $m=n$, then by Theorem 3.3,
$$xf_2'(G;x)=\sum_{e\in E(G)}f_2(G-e;x).$$
Given the initial condition $f_2(G;0)=0$, the differential equation above has a unique solution. So, combining with Theorem 4.1, the following theorem holds.

\begin{theorem}
The Laplacian  characteristic polynomial $f_2(G;x)=\det(xI-D+A)$ of a digraph $G$ with arc set $E(G)$ can be reconstructed from $\{f_2(G-e;x)|e\in E(G)\}$.
\end{theorem}

\section{Discussions}
In this paper, we consider mainly the problem about reconstructing characteristic polynomial $f_1(G;x)=\det(xI-A)$, Laplacian characteristic polynomial $f_2(G;x)=\det(xI-D+A)$, signless Laplacian characteristic polynomial $f_3(G;x)=\det(xI-D-A)$, permanental polynomial $f_4(G;x)={\rm per}(xI-A)$, Laplacian permanental polynomial $f_5(G;x)={\rm per}(xI-D+A)$, and signless Laplacian permanental polynomial $f_6(G;x)={\rm per}(xI-D-A)$ of a digraph $G$ from the collection of characteristic polynomials, Laplacian characteristic polynomials, signless Laplacian characteristic polynomials, permanental polynomials, Laplacian permanental polynomials, and signless Laplacian permanental polynomials of the edge deck of $G$, respectively.

We find that it is very different from the reconstruction conjecture, the edge reconstruction conjecture and the reconstruction conjecture on the characteristic polynomial of a graph, which are still open, if $m\neq n$, then six digraph polynomials $f_1(G;x), f_2(G;x),\\f_3(G;x), f_4(G;x),f_5(G;x), f_6(G;x)$ of a digraph $G=(V(G),E(G))$ can be reconstructed from $\{f_1(G-e;x)|e\in E(G)\}, \{f_2(G-e;x)|e\in E(G)\}, \{f_3(G-e;x)|e\in E(G)\}, \{f_4(G-e;x)|e\in E(G)\}, \{f_5(G-e;x)|e\in E(G)\}, \{f_6(G-e;x)|e\in E(G)\}$, respectively. We also show that if $m=n$, then $f_2(G;x)$ can be reconstructed from $\{f_2(G-e;x)|e\in E(G)\}$, and $f_1(G;x)$ and $f_4(G;x)$ can not be reconstructed from $\{f_2(G-e;x)|e\in E(G)\}$ and $\{f_4(G-e;x)|e\in E(G)\}$. A natural question is: Characterize the digraphs $G$ such that $f_3(G;x), f_5(G;x)$ and $f_6(G;x)$ can reconstructed from $\{f_3(G-e;x)|e\in E(G)\}, \{f_5(G-e;x)|e\in E(G)\}$ and $\{f_6(G-e;x)|e\in E(G)\}$ if $m=n$, respectively.

\end{document}